\documentclass[11pt,twoside,a4paper]{article}
\usepackage[english]{babel}   
\usepackage{amsmath}
\usepackage{amssymb}
\usepackage{amsthm}
\usepackage{enumerate}
\usepackage{eucal}
\usepackage{rotating}
\usepackage{psfrag}
\usepackage{epsfig}
\usepackage{makeidx}
\usepackage{bbm}
\usepackage{framed}
\usepackage{graphicx}
\usepackage[all, cmtip]{xy}
\usepackage{multicol}
\usepackage[T1]{fontenc}
\usepackage[utf8]{inputenc}
\usepackage[bottom]{footmisc}
\usepackage{authblk}

\theoremstyle{definition}
\newtheorem{definition}{Definition}[section]
\newtheorem{theorem}[definition]{Theorem}
\newtheorem{proposition}[definition]{Proposition}
\newtheorem{lemma}[definition]{Lemma}
\newtheorem{corollary}[definition]{Corollary}
\theoremstyle{remark}
\newtheorem{remark}[definition]{Remark}
\newtheorem{example}[definition]{Example}

\newcommand{\R}{\ensuremath{\mathbb{R}}}
\newcommand{\N}{\ensuremath{\mathbb{N}}}

\newcommand{\hide}[1]{}

\makeindex

\title{Domination properties and extension of positive compact operators on pre-Riesz spaces}
\author[1]{Onno van Gaans\thanks{vangaans@math.leidenuniv.nl}}
\author[1]{Feng Zhang\thanks{zhangfeng.0631@163.com}}
\affil[1]{Mathematical Institute,
Leiden University,
P.O. Box 9512,
2300 RA Leiden,
The Netherlands}
\begin{document}

\maketitle

\abstract{This paper concerns positive domination property of compact operators on pre-Riesz spaces. The method is embedding the pre-Riesz space to the Riesz completion. It extends the order continuous norms in pre-Riesz spaces to  Riesz completions. The compactness of  third power of a positive operator is obtained in a pre-Riesz space which has an order unit.}



\section{Introduction}

A classical question of positive compact operators between Banach lattices is that domination property. Due to P.G. Dodds and D.H. Fremlin \cite{dodds1979compact}, it is established that for Banach lattices  $X$, $Y$ such that $X'$, $Y$ both have order continuous norms, if a positive operator $S\colon X\rightarrow Y$ is dominated by a compact operator $T\colon  X\rightarrow Y$, i.e. $0\le S\le T$, then $S$ is compact. This positive domination property is also proved by A.W. Wickstead \cite{wickstead1981extremal} in the situations that either $X'$ or $Y$ is atomic with an order continuous norm. 
Moreover, due to C.D. Aliprantis and O. Burkinshaw \cite{aliprantis1980positive}, if a positive operator on a Banach lattice is dominated by a compact operator, then its third power is a compact operator.
Domination properties of positive strictly singular operators on Banach lattices are also studied 
by J. Flores, F.L. Hernandez and P. Tradacete in \cite{flores2011domination}.

In this paper, we mainly consider the similar positive compact domination property in pre-Riesz spaces. Namely, we explore under which suitable conditions for pre-Riesz spaces $X$ and $Y$, we have that every positive operator dominated by a compact operator is compact.
We also address the question whether a similar result concerning the third power of the operator is true for operators on pre-Riesz spaces. 

To use the theory of pre-Riesz spaces and Riesz completions, one natural question is  how we extend an order continuous norm on a pre-Riesz space to  its Riesz completion. 
In Section \ref{ext-of-ocn}, we will settle this by providing the condition that the pre-Riesz space is pervasive  and Archimedean.

Section \ref{ext-comp-operator} is concerned with a unique extension of an operator defined on a pre-Riesz space to its Dedekind completion, based on norm denseness and by means of order continuous norms.

In Section \ref{Com-domi-inpreR}, we investigate how the positive domination property of compact operators introduced by P.G. Dodds and D.H. Fremlin in \cite{dodds1979compact}, and the theory of the third power compact operators by C.D. Aliprantis and O. Burkinshaw \cite{aliprantis1980positive} can be generalized to  pre-Riesz spaces.


\section{Preliminaries}\label{sec.preli}

Let $X$ be a real vector space, let $K\subset X$ be the positive cone ($x,y \in K$, $\lambda, \mu\in \mathbb{R^+}$ implies $\lambda x+\mu y \in K$, and $K\cap (-K)=\{0\}$). The partial order in $X$ is defined by $x\ge y$ if and only if $x-y\in K$.
For a finite subset $M$ of $X$ the set of all upper bounds of $M$ will denoted by $M^\mathrm{u}$.
By $X$ is \emph{Archimedean} we mean if for every $x,y \in X$ with $nx\le y$ for all $n\in\mathbb{N} $ one has $x\le 0$. $X$ is called \emph{directed} if for every $x,y\in X$, there exists  $z\in X$ such that $z\ge x$ and $z\ge y$. 
The space $(X,K)$ is directed if and only if the cone $K$ is \emph{generating}
in $X$, that is, $X = K-K$.
A linear subspace $D\subset X$ is called \emph{order dense} in $X$ if for every $x\in X$ we have $x=\inf \{y\in D\colon y\ge x\}$. 


Recall that partially ordered vector space $X$ is called a \emph{pre-Riesz space}
if for every $x,y,z\in X$ the inclusion $\{x+y,x+z\}^\mathrm{u}\subseteq\{y,z\}^\mathrm{u}$ implies 
$x\in K$.
A linear map $i\colon X\rightarrow Y$, where $X$ and $Y$ are partially ordered vector spaces, is called \emph{bipositive} if for every $x\in X$ one has $0\le x$
 if and only if $0\le i(x)$. Then one can always 
  embedding a pre-Riesz space as an order dense subspace of a vector lattice by the following theorem, which is due to \cite[Corollaries 4.9-11 and Theorems 3.5, 3.7, 4.13]{Haa1993}.

\begin{theorem}\label{embd-preR}
Let $X$ be a partially ordered vector space. The following statements are equivalent.
\begin{itemize}
\item[(i)] $X$ is a pre-Riesz space.
\item[(ii)] There exist a vector lattice $Y$ and a bipositive linear map $i\colon X\rightarrow Y$ such that $i(X)$ is order dense in $Y$.
\item[(iii)] There exist a vector lattice $Y$ and a bipositive linear map $i\colon X\rightarrow Y$ such that $i(X)$ is order dense in $Y$ and $i(X)$ generates $Y$ as a vector lattice, i.e. for every $y\in Y$ there are $a_1, \ldots, a_m, b_1,\ldots, b_n\in i(X)$ such that 
\[y=\bigvee_{j=1}^{m}a_j-\bigvee_{k=1}^{n}b_k.\]
\end{itemize}
\end{theorem}
A pair $(Y, i)$ as in (ii) is called a \emph{vector lattice cover} of $X$, and all spaces $Y$ as in (ii) are isomorphically determined as vector lattices, so we will say 
$(Y,i)$ \emph{the Riesz completion} of $X$, denoted by $X^\rho$.

In a partially ordered vector space $X$, a net $(x_\alpha)_{\alpha\in I}\subset X$ is said to \emph{order converges}, in short \emph{o}-converges, to $x\in X$ if there exists a net $(y_\beta)_{\beta\in J}\subset X$ such that for every $\beta\in J$ there is $\alpha_0$ such that for every $\alpha\ge \alpha_0$ we have $\pm (x_\alpha-x)\le y_\beta\downarrow 0$. We denote this by $x_\alpha\xrightarrow{ \mathrm{o} }x$. 
A seminorm $p$ on $X$ is said to be \emph{order continuous} if $x_\alpha\xrightarrow{ \mathrm{o} }x$ implies $p(x_\alpha-x)\rightarrow 0$. A norm $\|\cdot\|$ on  $X$ is called \emph{semimonotone} if there exists $C\in\R^+$ such that for every $x,y\in X$ with $0\le x\le y$ one has $\|x\|\le C \|y\|$.

\begin{definition}
Let $X$ be a directed partially ordered vector space with a seminorm $p$, then we say that $p$ is \emph{regular} if for every $x\in X$,
\begin{equation}\label{regul-norm}
p(x)=\inf\{p(y)\colon\, y\in X, -y\le x\le y\}.
\end{equation}
\end{definition}

Let $(X,\tau)$ be a topological ordered vector space. We will use $X'$ to denote the  \emph{topological dual} of $(X,\tau)$, and  $X^\sim$ to denote the \emph{order dual} of $X$. The topological dual $X'$ of a locally convex-solid Riesz space $(X,\tau)$ is an ideal in its order dual $X^\sim$,  \cite[Theorem 3.49]{AliBur1985}.
Consider a dual system $\langle X, X'\rangle$, a locally convex topology $\tau$ on $X$ is said to be \emph{consistent} with the dual system if the topological dual of $(X, \tau)$ is precisely $X'$.

\begin{definition}
Let $X$ be a Riesz space, and let $X'$ be an ideal of $X^\sim$ separating the points of $X$. Then the pair $\langle X, X'\rangle$, under its natural duality $\langle x, x'\rangle:=x'(x)$, is said to be a \emph{Riesz dual system}.
\end{definition}

Recall that a subset $A$ in a topological vector space $(X,\tau)$ is called \emph{$\tau$-totally bounded}, if for every $\tau$-neighborhood $V$ of zero there is a finite subset $\Phi$ of $A$ such that $A\subseteq \bigcup_{x\in \Phi}(x+V)=\Phi+ V$.

For a Riesz space $X$ and its order dual $X^\sim$, the \emph{absolute weak topology} on $X$ is defined by a collection of seminorms $p_f$ via the formula
\[p_f(x)=|f|(|x|), \quad x\in X, f\in X^\sim,\]
and it is denoted by $|\sigma|(X, X^\sim)$. 
For a nonempty subset $A$ of $X^\sim$, the \emph{absolute weak topology generated by $A$} on $X$ is the locally convex solid topology on $X$ generated by the seminorms $p_f$ defined via the formula 
\[p_f(x)=|f|(|x|), \quad x\in X, \, f\in A.\]

Also, recall that for a subset $D$ of a topological vector space $(X, \tau)$, the \emph{restriction topology}, in short $r$-topology, on $D$ is such that $U\subseteq D$ is $r$-open if and only if there exists  $V\subseteq X$ which is $\tau$-open and $U=V\cap X$. 
For a net $(x_\alpha)_{\alpha\in I}$ in $D$ and $x\in D$, we then have $x_\alpha\xrightarrow{ r }x$ in $D$ if and only if $x_\alpha\xrightarrow{ \tau }x$ in $X$.

\begin{definition}
A Banach lattice $X$ is said to be
\begin{itemize}
\item[(1)] an \emph{AL-space} if $\|x+y\|=\|x\|+\|y\|$ for all $x,y\in X^+$ with $x\wedge y=0$ and 
\item[(2)] an \emph{AM-space} if $\|x\vee y\|=\max\{\|x\|, \|y\|\}$ for all $x,y\in X^+$ with $x\wedge y=0$.
\end{itemize}
\end{definition}

\section{Extension of order continuous norms}\label{ext-of-ocn}

On vector lattices, regular seminorms and Riesz seminorms coincide, this result is due to  \cite[Theorem 3.40]{Gaa1999}.
Moreover, due to \cite[Theorem 3.43 and Corollary 3.45]{Gaa1999}, one can extend the seminorm on a pre-Riesz space in the following way.
\begin{theorem}\label{norm-extension}
Let $(X, K)$ be a directed partially ordered vector space with a seminorm $p$. Let $Y$ be a directed partially ordered vector space and $i\colon X\to Y$ a bipositive linear map, such that $i(X)$ is majorizing in $Y$.
 Define 
\begin{eqnarray}\label{regul-ext}
p_r(y):=\inf\{p(x)\colon\, x\in X \ \mbox{such that} -i(x)\le y\le i(x)\}, \ y\in Y.
\end{eqnarray} 
The following statements hold.
\begin{itemize}
\item[(i)] $p_r$ is the greatest regular seminorm on $Y$ with $p_r\le p$ on $K$.
\item[(ii)] $p_r=p$ on $K$ if and only if $p$ is monotone. Moreover, if $p$ is monotone, then $p_r\ge \frac{1}{2}p$.
\item[(iii)] $p_r=p$ on $X$ if and only if $p$ is regular. 
\end{itemize}
\end{theorem}

Based on Zorn's lemma, it is shown that for a positive decreasing net in the Riesz completion $(X^\rho, i)$ of a pre-Riesz space $X$, if its infimum exists, then there exists a downward directed net in $i(X)^+$ such that their infimum are equal. This is due to \cite[Lemma 3.7.11]{anke-onno-book}.

\begin{lemma}\label{Lemma 3.7.11anke-onno-book}
Let $X$ be a pervasive Archimedean pre-Riesz space and $(Y, i)$ a vector lattice cover of $X$. Let $(y_\alpha)_{\alpha\in I}$ be a net in $Y$ such that $y_\alpha\downarrow 0$. There exists a net $(x_\beta)_{\beta\in J}$ in $X$ with $x_\beta\downarrow 0$ and such that for every $\beta\in J$ there exists $\alpha_0\in I$ such that for every $\alpha\ge \alpha_0$ we have $i(x_\beta)\ge y_\alpha$.
\end{lemma}

This lemma will be used to prove the following theorem, which extends order continuous norms on pre-Riesz spaces.

\hide{

\begin{remark}\label{dow-dir-con}
Clearly,  
for $x=\bigvee_{k=1}^{n}i(a_{k})$ in $(X^\rho)^+$,  and $z\in i(X)$ with $z-x\in (X^\rho)^+$, by the proof of Theorem \ref{preR-netconv},
 one has a net $(u_\alpha)_{\alpha\in D}$ in $i(X)^+$ which is upward directed and convergent to $z-x$. Define $v_\alpha=z-u_\alpha$ for every $\alpha\in D$, then  the net $(v_\alpha)_{\alpha\in D}\subseteq i(X)$ is downward directed and convergent to $x$.

\end{remark}

}

\begin{theorem}\label{order-cont-norm}
Let $X$ be a pervasive Archimedean pre-Riesz space,   $p$ a semimonotone seminorm on $X$, and $(Y, i)$ a vector lattice cover of $X$. Let $p_r$ on $Y$ be defined by 
\[p_r(y):=\inf\{p(x)\colon\, x\in X \ \mbox{such that} -i(x)\le y\le i(x)\}, \, y\in Y .\]
If $p$ on $X$ is order continuous, then  $p_r$ is an order continuous seminorm on $Y$ as well.
\end{theorem}

\begin{proof}
By Theorem \ref{norm-extension} (i) it follows that $p_r$ is a seminorm on $Y$. We show that $p_r$ is order continuous.

Assume that $x_\alpha\xrightarrow{o}x$ in $Y$, i.e. there exists a net $(y_\beta)_{\beta\in B}\subset Y$ such that for every $\beta$ there is $\alpha_0$ such that for every $\alpha\ge \alpha_0$ we have  $\pm (x_\alpha -x)\le y_\beta\downarrow0$.
As by Theorem \ref{norm-extension} (i) the seminorm $p_r$ is regular we have $p_r(x_\alpha-x)\le p_r(y_\beta)$.
By Lemma \ref{Lemma 3.7.11anke-onno-book} there exists a net $(z_\gamma)_\gamma\subset X$ such that  for every $\gamma$ there is $\beta_0$ such that for every $\beta\ge \beta_0$ we have
 $i(z_\gamma)\ge y_\beta$ and $z_\gamma\downarrow 0$. Since $p$ is order continuous, $p(z_\gamma)\downarrow0$.
 Since $p$ is semimonotone, there is a constant $C$ such that $p_r(i(v))\le Cp(v)$ for every $0\le v\in X$. 
  Then $p_r(x_\alpha-x)\le p_r(y_\beta)\le p_r(i(z_\gamma))\le Cp(z_\gamma)\rightarrow0$. Hence $p_r$ is order continuous on $Y$.
\end{proof}

The following result due to J. van Waaij \cite[Theorem 4.15, Corollary 4.16]{van2013tensor}. It will be used to prove the extension result of order continuous norm on the Dedekind completion of pre-Riesz spaces.

\begin{proposition}\label{J. Waaij}
For an Archimedean pre-Riesz space $X$, let $(X^\rho, i)$ be the Riesz completion. The following are equivalent. 
\begin{itemize}
\item[(i)] $X$ is pervasive.
\item[(ii)] For all $0< y\in X^\rho$, it holds
$y= \sup \{x\in i(X)\colon \, 0< x\le y \}$.
\end{itemize}
\end{proposition}

\begin{corollary}\label{ore-con-nor}
Let $X$ be a pervasive Archimedean pre-Riesz space  with an order continuous semimonotone seminorm $\|\cdot\|_X$. Let $(X^\delta, j)$ be the Dedekind  completion of $X$, and define  $\|\cdot\|_{X^\delta}$ on $X^\delta$ by
\[\|y\|_{X^\delta}:=\inf\{\|x\|_X\colon\, x\in X \ \mbox{such that} -j(x)\le y\le j(x)\}.\]
Then $\left\|\cdot\right\|_{X^\delta}$ is an order continuous seminorm and $\left(j(X), \|\cdot\|_X\right)$ is norm dense in $(X^\delta, \|\cdot\|_{X^\delta})$. Furthermore, $\|\cdot\|_{X^\delta}$ extends $\|\cdot\|_X$ if (and only if) $\|\cdot\|_X$ is regular.
\end{corollary}

\begin{proof}
The order continuity of $\|\cdot\|_{X^\delta}$ is same with  Theorem \ref{order-cont-norm}.
It remains to show that $(j(X), \|\cdot\|_X)$ is norm dense in $(X^\delta, \|\cdot\|_{X^\delta})$. Let $(X^\rho, i)$ be the Riesz completion of  $X$. For every $y\in X^\rho$, according to Lemma \ref{Lemma 3.7.11anke-onno-book},  there is a net $(x_\alpha)$ in $i(X)$ with $x_\alpha\xrightarrow{o}y$. By Theorem \ref{order-cont-norm}, $\|\cdot\|_{X^\rho}$ is order continuous, and then $\|x_\alpha-y\|_{X^\rho}\rightarrow0$, so that $i(X)$ is norm dense in $X^\rho$. Let $y\in (X^\delta)^+$, and $D:=\{x\in X^\rho\colon \, 0\le x\le y\}$. Since $X^\rho$ is pervasive in $X^\delta$ and $X^\rho$ is Archimedean, hence, by Proposition \ref{J. Waaij}, we have $D$ is upward directed and $\sup D= y$. Define the net $(z_\mu)_\mu$ by 
 $z_\mu:=\mu, \, \mu\in D$. Therefore $z_\mu\xrightarrow{o}y$ and $\|z_\mu-y\|_{X^\delta}\rightarrow0$. Thus we conclude  $X^\rho$ is norm dense in $X^\delta$. It follows that $j(X)$ is norm dense in $X^\delta$. 

Moreover, by Theorem \ref{norm-extension} (iii) the seminorm $\|\cdot\|_{X^\delta}$ extends $\|\cdot\|_X$ if (and only if) $\|\cdot\|_X$ is regular.
\end{proof}

\section{Extension of compact operators}\label{ext-comp-operator}

In this section, we will show that a compact operator on a pre-Riesz space $X$ with a suitable order continuous norm can be extended to a compact operator on the Dedekind completion $X^\delta$ of $X$. 
To make sure that the norm on $X^\delta$ indeed is an extension of the order continuous norm on $X$, the norm of $X$ is required to be regular. Then we will use Corollary \ref{ore-con-nor} to extend the operator, see the following theorem. 


\begin{theorem}\label{compact-ext}
Let $(X, \|\cdot\|_X)$ be a pervasive Archimedean pre-Riesz space  equipped with an order continuous regular seminorm. Let $(X^\delta, i)$ be the Dedekind completion of $X$.
Let $Y$ be a  Banach lattice with an order continuous norm. 
 If $T$ is a bounded operator in $L(X,Y)$, then there exists a unique bounded linear extension $\widehat{T}\in L(X^\delta,Y)$. If $T$ is compact, then $\widehat{T}$ is compact as well.
\end{theorem}
\begin{proof}
Recall that by  Theorem \ref{norm-extension} (i) there exists the greatest regular seminorm $\|\cdot\|_{X^\delta}$ on $X^\delta$ which extends $\|\cdot\|_X$. 
Due to Corollary \ref{ore-con-nor}, $i(X)$ is dense with respect to the seminorm in $X^\delta$. So  there exists $\widehat{T}\in L(X^\delta,Y)$ which uniquely extends $T$ by means of  for $z\in X^\delta$, $\widehat{T}z=\lim \widehat{T}(i(z_n))=\lim T(z_n)$, where $i(z_n)\subseteq i(X)$ norm converges to $z$.

Let $(x_n)_n$ be a norm bounded sequence in $X^\delta$. It follows from the norm denseness of $i(X)$ in $X^\delta$ that there exists a norm bounded sequence $(y_n)_n$ in $X$ such that 
$\|i(y_n)-x_n\|_{X^\delta}<\frac{1}{n}$ holds for all $n\in \N$. 
As $T$ is compact, $(Ty_n)_n$ has a convergent subsequence $(Ty_{n_k})_{n_k}$.
So there exists $y\in Y$ such that for $k\rightarrow \infty$ we have 
\[\|Ty_{n_k}-y\|_Y \rightarrow 0.\]
For the subsequence $\left(\widehat{T}x_{n_k}\right)$ of $\left(\widehat{T}x_n\right)$, one has
\begin{align*}
\left\|\widehat{T}x_{n_k}-Ty_{n_k}\right\|_{Y}
&=\left\|\widehat{T}x_{n_k}-\widehat{T}(i(y_{n_k}))\right\|_Y\\
&\leq\left\|\widehat{T}\right\|\left\|x_{n_k}-i(y_{n_k})\right\|_{X^\delta}\\
&\leq\left\|\widehat{T}\right\|\frac{1}{n_k}\rightarrow 0
\end{align*}
as  $ k\rightarrow \infty$. Then
\[
\left\|\widehat{T}x_{n_k}-y\right\|_{Y}\leq\left\|\widehat{T}x_{n_k}-Ty_{n_k}\right\|_{Y}+\left\|Ty_{n_k}-y\right\|_Y\rightarrow 0\]
as  $ k\rightarrow \infty$. So $\left(\widehat{T}x_n\right)$ has a convergent subsequence, and hence $\widehat{T}$ is compact.
\end{proof}

\section{Compact domination results in pre-Riesz spaces}\label{Com-domi-inpreR}

In this section, we will study two results concerning the  domination property of positive compact operators on the setting of pre-Riesz spaces. We will consider appropriate norms and use  Riesz completions or Dedekind completions of  pre-Riesz spaces.

The classical domination property of compact operators between Banach lattices 
by \cite[Theorem 5.20]{AliBur1985}, which  reads as follows.

\begin{theorem}(Dodds-Fremlin )\label{Theorem 5.20AliBur1985}
Let $X$ be a Banach lattice, $Y$ a Banach lattice with $X'$ and $Y$ having order continuous norms. If a positive operator $S\colon  X\rightarrow Y$ is dominated by a compact operator, then $S$ is  a compact operator.
\end{theorem}

\begin{remark}
This result is proved originally by P.G. Dodds and D.H. Fremlin in \cite{dodds1979compact}. An easier accessible proof is given in 
\cite[Theorem 5.20]{AliBur1985}. In fact, note that in the proof of \cite[Theorem 5.20]{AliBur1985}, it does not use the norm completeness of $X$. As a consequence, we have the following corollary.
\end{remark}

\begin{corollary}\label{dom-com-Rie}
Let $X$ be a normed Riesz space, $Y$ a Banach lattice with $X'$ and $Y$ having order continuous norms. If a positive operator $S\colon  X\rightarrow Y$ is dominated by a compact operator, then $S$ is  a compact operator.
\end{corollary}

Now our aim comes down to  extending Corollary \ref{dom-com-Rie} to pre-Riesz spaces. We firstly consider the domain space $X$ to be a pre-Riesz space, and we  suppose it is pervasive and  Archimedean. Alternatively, one could consider the order continuous regular seminorm on $X$, then use Theorem \ref{compact-ext} and Corollary \ref{dom-com-Rie} to obtain compactness of $S$.

\begin{theorem}\label{comp-domination1}
Let  $X$ be a pervasive Archimedean pre-Riesz space  equipped with an order continuous regular norm $\|\cdot\|_X$, $(X^\delta, i)$ the Dedekind completion of $X$.
 Let $Y$ be a   Banach lattice with an order continuous norm. 
Let $T$ in $L(X,Y)$ be a positive compact operator  and assume that  $(X^\delta)'$ has an order continuous norm. 
If $S\in L(X,Y)$ with $0\le S\le T$, then $S$ is compact.
\end{theorem}
\begin{proof}
Recall that by Theorem \ref{norm-extension} (i) there exists the greatest regular seminorm $\|\cdot\|_{X^\delta}$ on $X^\delta$ defined by the formula \eqref{regul-ext}, which extending the norm on $X$.
By Theorem \ref{compact-ext} there exists a unique bounded linear extension $\widehat{T}$ of $T$ on $X^\delta$, which is compact. 
In fact, for $x\in X^\delta$, $\widehat{T}$ is given by $\widehat{T}x=\lim \widehat{T}(i(x_n))=\lim Tx_n$, where $i(x_n)_n\subseteq i(X)$ norm converges to $x$. We define $\widehat{S}$ in a similar way.  So for $x\in (X^\delta)^+$ we have $(x_n)_n\subseteq X^+$ and $Sx_n\le Tx_n$ for all $n\in \N$. It follows that
\[\widehat{S}x=\lim Sx_n\le \lim Tx_n=\widehat{T}x.\] 
The positivity of $\widehat{S}$ is clear. It is clear that $\|\cdot\|_{X^\delta}$ is monotone. Moreover, it is a norm. In fact, let $z\in X^\delta$ and $z\neq 0$. Since $X$ is pervasive, there exists $y\in X$ with $0< i(y)\le |z|$. Then $\| z \|_{X^\delta}= \| |z| \|_{X^\delta}\ge\|i(y)\|_{X^\delta}$.  By Theorem \ref{norm-extension} (iii) we have $\|y\|_Y=\|i(y)\|_{Y^\delta}$ for every $y\in Y$. Hence, $\| z \|_{X^\delta}>0$. Hence, $\|\cdot\|_{X^\delta}$ is a Riesz norm. 

Thus we could use  Corollary \ref{dom-com-Rie} to conclude that $\widehat{S}$ is compact. Hence $S=\widehat{S}|_X$ is compact.
\end{proof}

For an operator between two pre-Riesz spaces $X$ and $Y$, we could extend it to the Riesz completion. To use the method in the proof of Theorem \ref{Theorem 5.20AliBur1985}, we also need that the norms of $(X^\delta)'$ and  $Y^\delta$ are order continuous. But it is difficult to characterize the space structure of dual spaces of pre-Riesz spaces, for example, whether $(X')^\rho$ equals $(X^\rho)'$ or not. Fortunately, if we suppose that the codomain space $Y$ is a directed Archimedean partially ordered vector space and complete with respect to the regular norm, then it has a Dedekind completion which is norm complete as well. Then we can 
embed  $Y$ into the Dedekind completion, and use Theorem \ref{comp-domination1}. See the following two theorems.

\begin{theorem}\label{dedekind-norm-comp}
Let $X$ be a directed Archimedean partially ordered vector space with a regular norm $\|\cdot\|$ such that $X^+$ is $\|\cdot\|$-closed. Let $(X^\delta, i)$ be the Dedekind completion of $X$, and define for $y\in X^\delta$,
\[\|y\|_r :=\inf\{\|x\|\colon \, x\in X^+, -i(x)\le y\le i(x) \}.\]
If $(X, \|\cdot\|)$ is complete, then $(X^\delta, \|\cdot\|_r)$ is complete.
\end{theorem}

\begin{proof}
The proof is similar to \cite[Theorem 2.12]{Gaa1999}. 
Assume that $(X, \|\cdot\|)$ is complete.
Let $(y_n)_n$ be a sequence in $X^\delta$ such that $\sum_{n=1}^{\infty}2^n\|y_n\|_r<\infty$. We show that there exists $y\in X^\delta$ such that $\left\|y-\sum_{n=1}^N y_n\right\|_r\rightarrow 0$ as $N\rightarrow\infty$. 
Take $x_n\in X^+$ with $-i(x_n)\le y_n\le i(x_n)$ and $\|y_n\|_r\ge \|x_n\|-\frac{1}{4^n}$. 
Then $\sum_{n=1}^{\infty}2^n\|x_n\|<\infty$. Since $X$ is norm complete, $u:=\sum_{n=1}^{\infty}2^n x_n$ exists in $X$. 
As for every $n\in \N$ we have $x_n\in X^+$ and $X^+$ is closed, we have $u\in X^+$ and $u\ge \sum_{n=1}^{m}2^n x_n\ge 2^n x_n$ for every $m\in \N$. 
Now for $N>M$, we have $\sum_{n=1}^{N}y_n-\sum_{n=1}^{M}y_n=\sum_{n=M+1}^{N}y_n\le \sum_{n=M+1}^{N}i(x_n)\le \left(\sum_{n=M+1}^{N}2^{-n}\right)i(u)$, and $\sum_{n=1}^{N}y_n-\sum_{n=1}^{M}y_n\ge -\left(\sum_{n=M+1}^{N}2^{-n}\right)i(u)$. 
Since $\sum_{n=M+1}^{N}2^{-n}\rightarrow 0$ when $M, N\rightarrow \infty$, we have that $\left(\sum_{n=1}^N y_n\right)_N$ is a relatively uniformly Cauchy sequence. 
Since $X^\delta$ is Dedekind complete and hence relatively uniformly complete, there exists $y\in X^\delta$ and there is a sequence $(\lambda_N)_N$ of reals with $\lambda_N\rightarrow 0$ such that $-\lambda_N i(u)\le y-\sum_{n=1}^{N}y_n\le \lambda_N i(u)$.
 Then $\left\|y-\sum_{n=1}^N y_n\right\|_r\le \lambda_N \|u\|\rightarrow 0$ as $N\rightarrow \infty$. Thus, $\left(X^\delta, \|\cdot\|_r\right)$ is complete.
\end{proof}

Then we will proof the positive domination property of compact operators on pre-Riesz spaces in the following theorem.

\begin{theorem}\label{pos-dom-propertyonpre-Riesz}
Let $X$ be a pervasive Archimedean pre-Riesz space equipped with an order continuous regular norm. Let $Y$ be a directed Archimedean pre-Riesz space with an order continuous regular norm $\|\cdot\|_Y$ that $Y$ is norm complete. Let $X^\delta$ be the Dedekind completion of $X$ and $(X^\delta)'$ has order continuous norm. 
If a positive operator $S\colon  X\rightarrow Y$ is dominated by $T$, i.e. $0\le S\le T$, and $T$ is compact, then $S$ is compact as well.
\end{theorem}
\begin{proof}
Let $Y^\delta$ be the Dedekind completion of $Y$, and $i\colon Y\rightarrow Y^\delta$ be the natural embedding map. Since $i$ is bipositive, we have $0\le i\circ S\le i\circ T$ form $X$ to $Y^\delta$. As $i$ is continuous and $T$ is compact,  $i\circ T$ is compact.
By Theorem \ref{norm-extension} (i) there exists the greatest regular seminorm  on $Y^\delta$ (in fact, since $Y$ has a norm, the seminorm on $Y^\delta$ is a norm) extending the norm on $Y$.
 Because $Y$ has order continuous norm, by Theorem \ref{order-cont-norm} the space $Y^\delta$ has order continuous norm as well, and by Theorem \ref{dedekind-norm-comp} the space $Y^\delta$ is norm complete. It follows from Theorem \ref{comp-domination1} that $i\circ S$ is compact. Let $(x_n)_n$ be a norm bounded sequence in $X$, then $(i\circ S)(x_{n})_n$ has a norm convergent subsequence $(i\circ S)(x_{n_k})_{n_k}$ in $Y^\delta$.
Then $(i\circ S)(x_{n_k})_{n_k}$ is a Cauchy sequence, so $S(x_{n_k})_{n_k}$ is a Cauchy sequence in $Y$. 
  Since $\|\cdot\|_Y$ is regular, by Theorem \ref{norm-extension} (iii) we have $\|y\|_Y=\|i(y)\|_{Y^\delta}$ for every $y\in Y$. As $Y$ is norm complete, $S(x_{n_k})$ is convergent in $Y$.
 Hence $S$ is compact. 
\end{proof}

By the following example, we show that a directed Archimedean pre-Riesz space which has an order continuous norm and complete with respect to this norm.

\begin{example} 
Let $X=\left\{x\in l^1\colon\, x^{(1)}=\sum_{k=2}^{\infty}x^{(k)}\right\}$.
Then $X$ is not a Riesz space, but a directed Archimedean pre-Riesz space. As $X$ is a closed subspace of $\left(l^1, \left\|\cdot\right\|_1\right)$, it is complete with respect to $\left\|\cdot\right\|_1$. Suppose that $(x_n)$ be a sequence in $X$ with  $x_n\downarrow 0$. From $\|\cdot\|_1$ is order continuous, it follows that $\left\|x_n\right\|_{1}=\sum_{k=1}^{\infty}x_n^{(k)}=2\sum_{k=2}^{\infty}x_n^{(k)} \le 2\sum_{k=1}^{\infty}x_n^{(k)}\rightarrow 0$.
\end{example}

Beside the domination property of compact operators  as in Theorem \ref{Theorem 5.20AliBur1985}, there is also a result in \cite[Theorem 5.13]{AliBur1985} which clarifies the positive domination property of third power of compact operators on Banach lattices, see the following theorem.

\begin{theorem}(Aliprantis-Burkinshaw)\label{Aliprantis-Burkinshaw}
If a positive operator $S$ on a Banach lattice is dominated by a compact operator, then $S^3$ is a compact operator.
\end{theorem}

We will show that similar result in pre-Riesz spaces by using the theory of AM-space, provided  that the pre-Riesz space has an order unit and an order unit norm. 

Let us recall some know results first.
The following two results can be found in \cite[Theorem 3.3 and Theorem 5.10]{AliBur1985}.

\begin{theorem}\label{con-tot-bounded}
Let $T\colon (X, \tau)\rightarrow (Y, \xi)$ be an operator between two topological vector spaces. If $T$ is continuous on the $\tau$-bounded subsets of $X$, then $T$ carries $\tau$-totally bounded sets to $\xi$-totally bounded sets.
\end{theorem}

\begin{theorem}(Dodds-Fremlin)\label{Dodds-Fremlin}
Let $X$ and $Y$ be two Riesz spaces with $Y$ Dedekind complete. If $\tau$ is an order continuous locally convex solid topology on $Y$, then for each $x\in X^+$, the set 
\[B=\{T\in L_b(X, Y)\colon \, T[0, x] \text{ is } \tau\text{-totally bounded}\}\]
 is a band in $L_b(X, Y)$.
\end{theorem}

The next result is due to S. Kaplan \cite[Theorem 3.50]{kaplan1959second}.

\begin{theorem}(Kaplan)\label{kaplan}
Let $X$ be a Riesz space, and let $A$ be a subset of $X^\sim$ separating the points of $X$. Then the topological dual of $(X, |\sigma|(X, A))$ is precisely the ideal generated by $A$ in $X^\sim$.
\end{theorem}

As a result of the above theorem, we have the following corollary.
\begin{corollary}\label{consist-dual-sys}
Let $X$ be a Riesz space and $\langle X, X'\rangle$ a Riesz dual system. Then the topological dual of $(X, |\sigma|(X, X'))$ is precisely $X'$, and $|\sigma|(X, X')$ is consistent with $\langle X, X'\rangle$.
\end{corollary}

Let us recall some known results with respect to the Riesz dual system. The following two are due to \cite[Theorem 3.57 and Theorem 3.54]{AliBur1985}.

\begin{theorem}\label{Theorem3.57AliBur1985}
For a Riesz dual system $\langle X, X'\rangle$, the following statements are equivalent.
\begin{itemize}
\item[(i)] $X$ is Dedekind complete and $\sigma(X, X')$ is order continuous.
\item[(ii)] $X$ is an ideal of $X''$.
\end{itemize}
\end{theorem}

\begin{theorem}\label{Theorem3.54AliBur1985}
For a Riesz dual system $\langle X, X'\rangle$, the following statements are equivalent.
\begin{itemize}
\item[(i)] Every consistent locally convex solid topology on $X$ is order continuous.
\item[(ii)] $\sigma(X, X')$ is order continuous.
\end{itemize}
\end{theorem}

Viewing the pre-Riesz space $X$ as an order dense subspace of its Riesz completion $(X^\rho, i)$. Our next goal is to establish the fact that if a positive operator from $X$ to $X$ is dominated by a compact operator, then it is continuous from the $r$-topology of $|\sigma|(X^\rho, (X^\rho)')$ to the norm topology.
We will use $x_\alpha\xrightarrow{w}x$ to denote that $x_\alpha$ converges to $x$ in $X$ with respect to the weak topology, i.e. the $\sigma(X, X')$-topology. See the following lemma.

\begin{lemma}\label{com-dom-continuous}
Let $X$ be an  Archimedean pre-Riesz space endowed with a monotone norm $\|\cdot\|_X$, $S, T\colon X\rightarrow X$ bounded linear operators satisfying $0\le S\le T$ with $T$  compact. Let $(X^\rho, i)$ be the Riesz completion of $X$ with the seminorm $\|\cdot\|_{X^\rho}$ given by \eqref{regul-ext}. Assume that there exists a positive compact operator $\widehat{T}\colon X^\rho\rightarrow X^\rho$ such that $\widehat{T}\circ i=i\circ T$. 
Then $S$ is continuous for the $r$-topology of $|\sigma|(X^\rho, (X^\rho)')$ on norm bounded subsets of  $X$ to the norm topology.
\end{lemma}
\begin{proof}
Let $(x_\alpha)_{\alpha}$ be a norm bounded net in $X$ with $x_\alpha\xrightarrow{r} 0$. It is enough to show that $\|S(x_\alpha)\|_X\rightarrow 0$ holds. 
To this end, from assumption of $x_\alpha\xrightarrow{r} 0$ in $X$ it follows that 
  $i(x_\alpha) \xrightarrow{|\sigma|(X^\rho, (X^\rho)')} 0$ in $X^\rho$, which means that for every $f\in (X^\rho)'$, we have $|f||i(x_\alpha)|\rightarrow 0$. It follows from $0\le |f(|i(x_\alpha)|)|\le |f||i(x_\alpha)|$ that $|i(x_\alpha)|\xrightarrow{w}0$ in $X^\rho$.

Since $\widehat{T}$ is compact and $(|i(x_\alpha)|)_\alpha$ is norm bounded, $\left(\widehat{T}(|i(x_\alpha)|)\right)_\alpha$ has a norm convergent subnet $\left(\widehat{T}y_\beta\right)_\beta$ such that $\widehat{T}y_\beta\xrightarrow{\|\cdot\|_{X^\rho}}z$ for some $z\in X^\rho$, and then $\widehat{T}y_\beta\xrightarrow{w}z$. 
As for every $f\in (X^\rho)'$ we have $\langle f, \widehat{T}(|i(x_\alpha)|)\rangle= \langle \widehat{T}^*f, |i(x_\alpha)|\rangle\rightarrow 0$. So $\widehat{T}(|i(x_\alpha)|)\xrightarrow{w}0$ and then $z=0$. Hence $\left(\widehat{T}(|i(x_\alpha)|)\right)_\alpha$ has a subnet which converges in norm to $0$. 
Similarly, every subnet of $\left(\widehat{T}(|i(x_\alpha)|)\right)_\alpha$ has a subnet that norm converges to $0$. Therefore, $\left\|\widehat{T}(|i(x_\alpha)|)\right \|_{X^\rho}\rightarrow 0$.

Since $0\le S\le T$ and $i$ is positive, we have $0\le \widehat{S}\circ i\le \widehat{T}\circ i$. According to Theorem \ref{norm-extension} (ii) we have that $\|x\|_X\le 2\|i(x)\|_{X^\rho}$ for every $x\in X$. Thus, we have 
\[\|S(x_\alpha)\|_X\le 2\left\|(\widehat{S}\circ i)(x_\alpha)\right\|_{X^\rho}\le 2\left\|\widehat{S}(|i(x_\alpha)|)\right\|_{X^\rho}\le 2\left\|\widehat{T}(|i(x_\alpha)|)\right\|_{X^\rho}\rightarrow 0.\]
Hence $S$ is continuous. 
\end{proof}

We continue with a approximation property on the Riesz completion of a pre-Riesz space with respect to the regular norm.

\begin{lemma}\label{lemma 5.1}
Let  $X, Y$ be two normed pre-Riesz spaces with a regular norm on $Y$, and $S,T\colon X\rightarrow Y$ such that $0\le S\le T$. Let $(Y^\rho, i)$ be the Riesz completion of $Y$.
If $T$ sends a subset $A$ of $X^+$ to a norm totally bounded set, then for each $\epsilon > 0$ there exists some $u\in (Y^\rho)^+$ such that for all $x\in A$ we have
\[\left\|(i(Sx)-u)^+\right\|_{Y^\rho}\le \epsilon.\]
\end{lemma}
\begin{proof}
By Theorem \ref{norm-extension} one can define the seminorm on $Y^\rho$ by 
 \[\|y\|_{Y^\rho}:=\inf\{\|x\|_Y\colon \, x\in Y, -i(x)\leq y\leq i(x)\}, \,  y\in Y^\rho.\]
As $\|\cdot\|_Y$ is regular, it is obvious that $\|i(y)\|_{Y^\rho}= \|y\|_Y$ for $y\in Y$.
Let $\epsilon>0$. Since $T$ sends a subset $A$ of $X^+$ to a norm totally bounded set of $Y$,  there exist $x_1,\ldots, x_n\in A$ such that for all $x\in A$ we have $\|Tx-Tx_j\|_Y<\epsilon$ for some $j$.
Put $u=(i\circ T)\left(\sum_{j=1}^n x_j\right)\in (Y^\rho)^+$, then for $x\in A$ and $j$ with $\|Tx-Tx_j\|_Y<\epsilon$ we have
\begin{align*}
0\le (i(Sx)-u)^+
&=\left(i\bigl(Sx-T\sum_{j=1}^n x_j\bigr)\right)^+\\
&\le \left(i \bigl(Tx-T\sum_{j=1}^n x_j\bigr)\right)^+ \\
& \le \bigl(i (Tx-Tx_j)\bigr)^+\\
& \le \left|\bigl(i (Tx-Tx_j)\bigr)^+\right|.
\end{align*}
Since regular norms are monotone we have 
\begin{align*}
0\le \left\|(i(Sx)-u)^+\right\|_{Y^\rho}
&\le \left\| |i(Tx-Tx_j)| \right\|_{Y^\rho}\\
&=\left\|i(Tx-Tx_j) \right\|_{Y^\rho}\\
&=\left \|Tx-Tx_j\right\|_Y\le \epsilon.
\end{align*}
Thus we have completed the proof.
\end{proof}

Next, we will show that for every Archimedean Riesz space $X$ with an order unit $u$ and an order unit norm $\|\cdot\|_u$, $\langle X'', X'\rangle$ is a Riesz dual system. To this end we need to show that $X'$ is an ideal of $(X'')^\sim$ and  use the theory of AL-spaces and AM-spaces.

The next lemma can be found in \cite[Proposition 1.4.7]{Mey1991}.

\begin{lemma}\label{X'-ALspace}
Let $X$ be an Archimedean Riesz space with an order unit $u$ and an order unit norm $\|\cdot\|_u$. Then $(X', \|\cdot\|)$ is an AL-space.
\end{lemma}
The following result is from \cite[Corollary 3.7]{abramovich2002invitation}.

\begin{corollary}\label{Cor3.7abramovich2002invitation}
Every AL-space has order continuous norm.
\end{corollary}
The next result is due to Nakano, see \cite[Theorem 4.9]{AliBur1985}.
\begin{theorem} \label{Th4.9AliBur1985}
For a Banach lattice $X$ the following statements are equivalent.
\begin{itemize}
\item[(i)] $X$ has order continuous norm.
\item[(ii)] $X$ is an ideal of $X''$.
\end{itemize}
\end{theorem}
We cite the following result by G. Birkhoff, see \cite[Corollary 4.5]{AliBur1985}.
\begin{corollary}\label{Cor4.5AliBur1985}
The norm dual of a Banach lattice $X$ coincides with its order dual, i.e., $X'=X^\sim$.
\end{corollary}

\begin{lemma}\label{Riesz-dual-sys}
Let $X$ be an Archimedean Riesz space with an order unit norm. Then $\langle X'', X'\rangle$ is a Riesz dual system.
\end{lemma}
\begin{proof}
By Lemma \ref{X'-ALspace}, $X'$ is an AL-space. By Corollary \ref{Cor3.7abramovich2002invitation}, every AL-space has an order continuous norm, so by Theorem \ref{Th4.9AliBur1985} $X'$ is an ideal of $X'''$. Since $X''$ is a Banach lattice, it follows from  Corollary \ref{Cor4.5AliBur1985} that $X'''=(X'')^\sim$. Hence $X'$ is an ideal in $(X'')^\sim$. 
Moreover, let $x \in X''$ with $x\neq 0$. Then there exists some $f\in X'$ satisfying $f(x)=x(f)\neq 0$. 

Hence $\langle X'', X'\rangle$ is a Riesz dual system.
\end{proof}

Since a pre-Riesz space with an order unit and an order unit norm need not be norm complete,  we  can not use \cite[Theorem 5.11]{AliBur1985} directly (in there it is required for the spaces to be Banach lattices). 
The next lemma is a modification of \cite[Theorem 5.11]{AliBur1985},  provided that the range space has an order unit and an order unit norm.

\begin{lemma}\label{totally-bdd-domi}
Let $X$, $Y$ be two normed Riesz spaces such that $Y$ has an order unit and equipped with an order unit norm.
Let $S, T\colon X\rightarrow Y$ be two positive operators with $0\le S\le T$. If $T[0,x]$ is $|\sigma|(Y, Y')$-totally bounded for each $x\in X^+$, then $S[0, x]$ is likewise $|\sigma|(Y, Y')$-totally bounded for each $x\in X^+$.
\end{lemma}

\begin{proof}
By Lemma \ref{Riesz-dual-sys}, the pair  $\langle Y'', Y'\rangle$ is a Riesz dual system.  Since $Y''$ is an ideal of $Y''$,  it follows from Theorem \ref{Theorem3.57AliBur1985} that $\sigma(Y'', Y')$ is an order continuous topology on $Y''$.
By Theorem \ref{Theorem3.54AliBur1985} we therefore have that every consistent locally convex solid topology on $Y''$ is order continuous.
Observe that Corollary \ref{consist-dual-sys} it yields that $|\sigma|(Y'', Y')$ is consistent, and then  $|\sigma|(Y'', Y')$
 is an order continuous locally convex solid topology.

Let us view $S, T$ as two operators from $X$ to $Y''$.
As $T[0,x]$ is $|\sigma|(Y, Y')$-totally bounded for each $x\in X^+$, it has $T[0,x]$ is $|\sigma|(Y'', Y')$-totally bounded for each $x\in X^+$. 
Since $Y$ equipped with an order unit norm, we have $Y'$ is Dedekind complete. 
Hence, it follows from Theorem \ref{Dodds-Fremlin} that $S[0,x]$ is $|\sigma|(Y'', Y')$-totally bounded for each $x\in X^+$. So $S[0, x]$ is  $|\sigma|(Y, Y')$-totally bounded for each $x\in X^+$.
\end{proof}

We are now in the position to extend the result of Theorem \ref{Aliprantis-Burkinshaw} to a setting of  pre-Riesz space.
To this end, first notice that for every order unit $e$ in a pre-Riesz space $X$ with the Riesz completion $(X^\rho, i)$, the element $i(e)$ is an order unit in $X^\rho$.
\begin{theorem}\label{S^3-dom-compact}
Let $X$ be an Archimedean pre-Riesz space with an order unit $e$ and the order unit norm $\|\cdot\|_X$ such that $X$ is norm complete. Let $S,T \colon X\rightarrow X$ satisfy $0\le S\le T$, and let $T$ be compact. Assume that there exists a positive compact operator $\widehat{T}\colon X^\rho\rightarrow X^\rho$ such that $\widehat{T}\circ i=i\circ T$, where $(X^\rho, i)$ is the Riesz completion of  $X$ equipped with the order unit norm $\|\cdot\|_{X^\rho}$. Then $S^3$ is compact.
\end{theorem}
\begin{proof}

Let $U=\{x\in X\colon \, \|x\|_X\le 1\}$ be the closed unit ball in $X$. Since the norm on $X$ is an order unit norm,  we have 
\[\|x\|_X=\inf \{\lambda\in \R^+\colon\, -\lambda e\le x\le \lambda e\}, \, x\in X.\]
Let $x\in U$ be such that $\|x\|_X\le 1$. 
Hence there exists $\lambda\le 1$ with $-y\le x\le y$ and $y=\lambda e$. Due to $x=\frac{1}{2}(y+x)-\frac{1}{2}(y-x)$  it follows from $0\le y+x\le 2y$ and $0\le y-x\le 2y$ 
that $\|y+x\|_X\le 2\|y\|_X\le 2$ and $\|y-x\|_X\le 2\|y\|_X\le 2$, respectively.  Thus $\frac{1}{2}(y+x), \frac{1}{2}(y-x)\in U\cap X^+$, and then $U\subseteq U^+-U^+$ holds. Therefore, it is enough to show that $S^3(U^+)$ is a norm totally bounded set.

It is clear that   $i(e)$ is an order unit in $X^\rho$, 
and the extension norm $\|\cdot\|_{X^\rho}$ is an order unit norm with respect to $i(e)$. By Lemma \ref{lemma 5.1} there exists some $u\in (X^\rho)^+$ such that $\|(i(Sx)-u)^+\|_{X^\rho}\le \epsilon$ for all $x\in U^+$. This implies that $0\le (i(Sx)-u)^+\le \epsilon i(e) $. Thus we have the following estimate,
\begin{align*}
i(Sx-\epsilon e)&=i(Sx)\wedge u + (i(Sx)-u)^+-\epsilon i(e)\\
& \le i(Sx)\wedge u\\
& \le u.
\end{align*} 
So $i(Sx)\in [0, u+\epsilon i(e)]$ for every $x\in U^+$. Take $v\in X$ with $i(v)\ge u$. Then for every $x\in U^+$ we have $i(Sx)\in [0, i(v+\epsilon e)]$. Hence 
\[S(U^+)\subseteq [0, v+\epsilon e]. \] 
Therefore 
\[
S^2(U^+)\subseteq S[0, v+\epsilon e], \quad \mbox{ and }  \quad  S^3(U^+)\subseteq S^2[0, v+\epsilon e]. 
\tag{$\ast$}
\] 
By Lemma \ref{com-dom-continuous} the operator $S$ is continuous on norm bounded subsets of $X$ with respect to the $r$-topology of the $|\sigma|(X^\rho, (X^\rho)')$ to the norm topology. 
By Theorem \ref{con-tot-bounded} the operator $S$ maps totally bounded sets with respect to the $r$-topology of $|\sigma|(X^\rho, (X^\rho)')$ to norm totally bounded sets. Since $T[0, v+\epsilon e]$ is norm totally bounded, $T[0, v+\epsilon e]$ is totally bounded with respect to the $|\sigma|(X^\rho, (X^\rho)')$ topology. Hence,  $T[0, v+\epsilon e]$ is totally bounded with respect to the $r$-topology of $|\sigma|(X^\rho, (X^\rho)')$. Then by Lemma \ref{totally-bdd-domi} the set $S[0,v+\epsilon e]$ is likewise totally bounded with respect to  the $r$-topology of $|\sigma|(X^\rho, (X^\rho)')$. Clearly, $S[0,v+\epsilon e]\subseteq [0, S(v+\epsilon e)]$.
Therefore,  $S^2[0,v+\epsilon e]=S(S[0,v+\epsilon e])$ is a norm totally bounded set. By the second inclusion of $(\ast)$ 
we have $S^3(U^+)$ is a norm totally bounded set, as desired.
\end{proof}

\begin{remark}
In the view of Theorem \ref{S^3-dom-compact}, it is of interest to know under which condition such a positive compact operator $\widehat{T}$ exists. By Theorem \ref{comp-domination1} and Theorem \ref{pos-dom-propertyonpre-Riesz}, we know that this is true for $X$ with order continuous norm.
However, if $X$ is a Banach lattice with an order unit $e$, then it can be renormed by 
\[\|x\|_\infty=\inf\{\lambda>0\colon \, |x|\le \lambda e\},\]
and $X$ becomes an AM-space. By Kakutani-Bohnenblust and M. Krein-S. Krein representation theorem \cite[Theorem 4.29]{AliBur1985},
we have $X$ is lattice isometric to some $C(\Omega)$ for a (unique up to homeomorphism) Hausdorff compact topological space $\Omega$, and the norm $\|\cdot\|_{\infty}$ on  $C(\Omega)$ is not order continuous.
So it is still an open question which choice of a norm on $X$ leads to a similar result as Theorem \ref{S^3-dom-compact}.
\end{remark}

\end{document}